\documentclass[11pt,reqno]{amsart}
\usepackage[T1]{fontenc}
\usepackage{fullpage,a4wide}
\usepackage{youngtab}
\usepackage{graphicx}
\usepackage{hyperref}
\hypersetup{colorlinks=true, linkcolor=black, citecolor=black, urlcolor=black}
\usepackage{amsmath,amsopn,amssymb,amsfonts,stmaryrd}
\usepackage{verbatim}
\usepackage{amsthm}
\usepackage{mathtools}
\usepackage{color}
\usepackage{enumitem}
\usepackage{bbm}
\usepackage{mathrsfs}
\usepackage{booktabs}
\usepackage{bm}
\usepackage{tensor}
\usepackage{cleveref}
\usepackage{float}

\newcommand{\N}{\mathbb N}

\theoremstyle{plain}
\newtheorem{theorem}{Theorem}[section]
\newtheorem{corollary}[theorem]{Corollary}
\newtheorem{lemma}[theorem]{Lemma}
\newtheorem{proposition}[theorem]{Proposition}

\newtheorem{challenge}{Challenge}

\theoremstyle{remark}
\newtheorem*{remark*}{Remark}
\newtheorem*{remarks*}{Remarks}

\theoremstyle{definition}

\newtheorem*{notation*}{Notation}
\newtheorem*{claim*}{Claim}
\newtheorem*{Conjecture1*}{Conjecture 1}
\newtheorem*{Conjecture2*}{Conjecture 2}
\newtheorem*{Conjecture3*}{Conjecture 3} 
\newtheorem*{Conjecture4*}{Conjecture 4}

\numberwithin{equation}{section}

\setlist[itemize]{noitemsep, topsep=0pt}

\makeatletter
\newcommand{\vast}{\bBigg@{3}}
\newcommand{\Vast}{\bBigg@{5}}
\makeatother

\allowdisplaybreaks

\def\func#1{\mathop{\rm #1}}%

\newcommand{\eisenst}{\mathrm{E}}
\newcommand{\hilfsftn}{F}
\newcommand{\links}{\eta }
\newcommand{\Mnhkoeff}{\left( M_{n}^{h}\right) }
\newcommand{\Mnhexp}{\nu }
\newcommand{\naeherung}{\delta }
\newcommand{\negmaxmod}{r}
\newcommand{\nop}{\mathrm{NO}}
\newcommand{\pol}{\mathrm{P}}
\newcommand{\rechts}{\zeta }
\newcommand{\residuum}{\Omega }
\newcommand{\spektral}{\Delta }
\newcommand{\zerlegung}{K}

\newcommand{\auxftn}{P}
\newcommand{\auxprodftn}{Q}
\newcommand{\coefauxprod}{q}
\newcommand{\colorset}{E}
\newcommand{\rectangle}{R}
\makeatletter
\@namedef{subjclassname@2020}{%
	\textup{2020} Mathematics Subject Classification}
\makeatother
\begin{document}
%%\title{Spectral Properties of\\ Nekrasov--Okounkov Polynomials} 
%%\title{A Perron--Frobenius Approach for Dominant Zeros of Nekrasov--Okounkov Polynomials}
%%\title{Dominant Zeros of Nekrasov--Okounkov Polynomials via Perron--Frobenius Theory}
\title{Dominant Zeros of Nekrasov--Okounkov Polynomials}
%%{Deformation of the Nekrasov-Okounkov Polynomials and the Zeros of Largest Modulus}
\author{Bernhard Heim}
\address{Department of Mathematics and Computer Science\\Division of Mathematics\\University of Cologne\\ Weyertal 86--90 \\ 50931 Cologne \\Germany}
\address{Lehrstuhl A f\"{u}r Mathematik, RWTH Aachen University, 52056 Aachen, Germany}
\email{bheim@uni-koeln.de}
\email{bernhard.heim@rwth-aachen.de}
\author{Markus Neuhauser}
\address{Kutaisi International University, 5/7, Youth Avenue,  Kutaisi, 4600 Georgia}
\address{Lehrstuhl
f\"{u}r Geometrie und
Analysis, RWTH Aachen University, 52056 Aachen, Germany}
\email{markus.neuhauser@kiu.edu.ge}
%%%
\dedicatory{\textsc{With Appendix~\ref{appxB} by Ken Ono}}
%\linenumbers
%%%
\begin{abstract}
We give an exact finite-dimensional Perron--Frobenius realization of the dominant zero of the Nekrasov--Okounkov polynomials $\nop _n(z)$.
For a normalized positive sequence $h=(h(n))_{n\ge 1}$ with $h(1)=1$, define $\pol _0^h(z)=1$ and, for $n\ge 1$,
\[ \pol _n^h(z)=\frac{z}{h(n)}\sum_{k=1}^n \sigma(k)\pol _{n-k}^h(z),\]
where $\sigma(k)$ denotes the sum of divisors of $k$. The Nekrasov--Okounkov polynomials are obtained from the specialization $h(n)=n$ by the shift $\nop _n(z)=\pol _n^h(z+1)$. We derive a Hessenberg determinant representation for $\pol _n^h(z)$. After separating the trivial zero at the origin, the remaining zeros of $\pol _n^h(-z)$ are identified with the eigenvalues of an explicit $(n-1)\times(n-1)$ nonnegative matrix $M_n^h$. We prove that $M_n^h$ is primitive and apply Perron--Frobenius theory to show that $\pol _n^h(z)$ has a unique zero of maximal modulus; this zero is real, negative, and simple. As a consequence, the same property holds for the Nekrasov--Okounkov polynomials. We also prove strict monotonicity of the associated spectral radii.
\end{abstract}
%%%%%%%%%%%%%%%%%%%%%%%%%%%%%%%%%%%%%%%%%%%%%%%%%%%%%%%%%%%%%%%%%%%%%%%%%%%%%%%%%%%%%%%%%%%%%%%%%%%%%%%%%%
%%%%%%%%%%%%%%%%%%%%%%%%%%%%%%%%%%%%%%%%%%%%%%%%%%%%%%%%%%%%%%%%%%%%%%%%%%%%%%%%%%%%%%%%%%%%%%%%%%%%%%%%%%
%%%%%%%%%%%%%%%%%%%%%%%%%%%%%%%%%%%%%%%%%%%%%%%%%%%%%%%%%%%%%%%%%%%%%%%%%%%%%%%%%%%%%%%%%%%%%%%%%%%%%%%%%%
%%%%%%%%%%%%%%%%%%%%%%%%%%%%%%%%%%%%%%%%%%%%%%%%%%%%%%%%%%%%%%%%%%%%%%%%%%%%%%%%%%%%%%%%%%%%%%%%%%%%%%%%%%
\maketitle
\setcounter{tocdepth}{1}
%%\tableofcontents
%%\medskip
\section{Introduction}

The Nekrasov--Okounkov hook-length formula occupies a remarkable position at the
interface of partition theory, modular forms, representation theory, and mathematical
physics. Originating in the work of Nekrasov and Okounkov on Seiberg--Witten theory
and random partitions \cite{NO06}, the identity relates weighted hook-length generating
functions to powers of Euler products and, ultimately, to powers of the Dedekind eta
function \cite{Ne55,On04}. The formula was independently discovered by Han \cite{Ha10} and
Westbury \cite{We06} and proved by several methods.
Let $\lambda\vdash n$ be a partition of $n$ %,
and let $H(\lambda)$
denote the multiset of hook lengths of $\lambda$. %, and let $\mathcal P$ denote the set of
%all partitions.
The Nekrasov--Okounkov hook-length formula is
\begin{equation*} \label{hook length}
\sum_{n=0}^{\infty} \nop _n(z)q^n
=
\sum_{\lambda %\in\mathcal P
}
q^{|\lambda|}
\prod_{h\in H(\lambda)}
\left(1+\frac{z}{h^2}\right)
=
\prod_{m=1}^{\infty}(1-q^m)^{-z-1}.
\end{equation*}
This identity belongs to a broader circle of results on hook lengths, contents, and partition statistics; see, for instance, Stanley \cite{St10}.
The corresponding Nekrasov--Okounkov polynomials $\nop _n(z)$ exhibit striking arithmetic
and analytic behavior.
%%%%%%%%%%%%%%%%%%%%%%%%%%%%%%
However, their zero distribution is only partially understood. In particular,
no general finite-dimensional structural explanation was previously available
for the distinguished zero of maximal modulus. This is the problem addressed
in the present paper.
%%%%%%%%%%%%%%%%%%%%%%%%%%%%%%%%%%%%%%%%%%%%%%%%%%%%%%%%%%%%%%%%% new
We focus on the spectral structure underlying these polynomials. Our starting
point is the observation that the polynomials $\nop _n(z)$ can be realized as
specializations of generalized D'Arcais polynomials \cite{DA13,HN20}. This
recursive formulation leads naturally to a Hessenberg determinant
representation. After a diagonal sign conjugation and the removal of the
trivial zero at the origin, the remaining zeros are identified with the
eigenvalues of an explicit primitive nonnegative matrix. Consequently, the
dominant zero is identified with the Perron zero of this matrix. This creates
a direct bridge between hook-length generating functions and the spectral
theory of nonnegative matrices \cite{BP94,HJ12,Se06}.

The novelty of the present approach is not merely the use of Perron--Frobenius theory, but the construction of the primitive nonnegative matrix itself. It converts a zero problem for partition-theoretic polynomials into a spectral-radius problem for an explicit Hessenberg matrix. This provides an exact finite-dimensional explanation of the distinguished real negative zero observed experimentally.

The study of zeros of partition-theoretic polynomials has a substantial history.
The D'Arcais polynomials of small degree appeared already in Serre's approach
\cite{Se85} to the lacunary property of powers of the Dedekind eta function for even powers.
Boyer and Goh studied partition polynomials and their asymptotic zero sets
\cite{BG08}, while Boyer and Parry investigated zero distributions for plane
partition polynomials and related zero-attractor phenomena \cite{BoPa12,BoPa21}.
More recently, questions related to D'Arcais and Nekrasov--Okounkov structures
have also appeared in the study of tuples of commuting permutations, coefficient
arrays, log-concavity, and asymptotic coefficient behavior; see the work by
Abdesselam, Brunialti, Doan, and Velie \cite{ABDV24}, Abdesselam \cite{Ab25},
Abdesselam and Starr \cite{AS26}, and recent preprints by Starr
\cite{St26a,St26b}. Related questions for hook polynomials and partition
statistics were studied by
Griffin, Ono, and Tsai \cite{GOT24},
Craig, Ono, and
Singh \cite{COS25},
Bridges, Craig, Folsom, and Rolen \cite{BCFR25},
and us \cite{HN19}.
The present paper is complementary to these approaches: rather than studying
asymptotic zero attractors or coefficient asymptotics, it gives an exact
finite-dimensional Perron--Frobenius approach for the dominant zero.

\section{Main Results}

Let $h=(h(n))_{n\ge1}$ be a normalized sequence of positive real numbers with $h(1)=1$.
The generalized D'Arcais polynomials $\pol _n^h(z)$ \cite{HNT20} are defined recursively by
\[
\pol _n^h(z)
=
\frac{z}{h(n)}
\sum_{k=1}^n
\sigma(k)\,\pol _{n-k}^h(z),
\qquad
\pol _0^h(z)=1,
\]
where $\sigma(k)$ denotes the sum of divisors of $k$. 
We mainly focus on
$h_s(n):=n^s$ with $s \in \mathbb{R}$. Then $s=1$ gives the D'Arcais polynomials, relevant for the Nekrasov--Okounkov polynomials.  For $s=0$, we have polynomials of Volterra type and a relation to quasi-modular forms: $$\sum_{n=0}^{\infty} \pol _n^{h_0}(24) q^n =\frac{1}{\eisenst _2(\tau)}.$$
Here $\eisenst _2(\tau)$ denotes the Eisenstein series of weight $2$ and $q=\mathrm{e}^{ 2 \pi \mathrm{i} \tau}$ where $\tau$ is in the upper complex
half plane \cite{On04, HN24}.

In this paper, we prove new results for dominant zeros of %the
generalized D'Arcais polynomials.

\begin{theorem}[Dominant zero for generalized D'Arcais polynomials]
\label{th:Darcais}
Let $h=(h(n))_{n\ge1}$ be a normalized sequence of positive real numbers, and let
$n\ge2$. Then the polynomial $\pol _n^h(z)$ has a unique zero of maximal modulus. This
zero is real, negative, and simple.
\end{theorem}

The theorem may be viewed as a Perron--Frobenius phenomenon for recursively defined
partition-type polynomials. The proof proceeds by realizing the zeros of $\pol _n^h(z)$ as
eigenvalues of an explicit Hessenberg matrix and then establishing primitivity after a suitable
diagonal sign conjugation and restriction. This allows the application of the classical Perron--Frobenius
theorem for primitive matrices \cite{BP94,HJ12,Se06}. 
This finite-dimensional realization is the key structural input in the proof.

It complements earlier zero-transfer and zero-estimate results for recursively defined polynomials given in \cite{HNT23,HN24B}.

As an immediate consequence we obtain the corresponding result for the Nekrasov--Okounkov
polynomials.

\begin{corollary}[Nekrasov--Okounkov polynomials]\label{NO}
Let $\nop _n(z)$ be the $n$th Nekrasov--Okounkov polynomial. Then $\nop _n(z)$ has a unique
zero of maximal modulus. This zero is real, negative, and simple.
\end{corollary}

This result contributes to the conjectural picture proposed in \cite{HN19} that the zeros of
Nekrasov--Okounkov polynomials are simple. Moreover, the theorem provides a conceptual
explanation for the appearance of distinguished negative real zeros observed experimentally.

A key ingredient in the proof is the following determinant representation of Hessenberg type. 
Let $c_0:=1$ and, for $n \geq 1$, define
 \begin{equation}\label{cm}
c_n := \sum_{k=0}^{n-1} (-1)^{n-k+1} \sigma(n-k+1) \,\, c_k, \quad  n \geq 1.
\end{equation}
%%%%%%%%%%
%%
For $n\ge2$, define the $(n-1)\times(n-1)$ matrix
\[
M_n^h=\left(\Mnhkoeff _{ij}\right)_{1\le i,j\le n-1},
\qquad
\Mnhkoeff _{ij}=
\begin{cases}
c_{i-j+1}h(j), & i\ge j-1,\\
%h(i+1), & j=i+1,\\
0, & j>i+1.
\end{cases}
\]

\begin{theorem}[Determinant representation] \label{th:Hessenberg}
For $n\ge2$ the generalized D'Arcais polynomials satisfy the following identity.
\begin{equation*}
\pol _n^h(-z)
= \frac{(-1)^n \, z}{
\prod_{k=1}^n h(k)} \,\,
\det\!\left(zI_{n-1}-M_n^h\right).
\end{equation*}
The matrix $M_n^h$ has nonnegative entries.
\end{theorem}
For instance, for $n=4$ the matrix occurring in Theorem 2.3 is
\[
M^h_4=
\begin{pmatrix}
 c_1 & h(2)&0\\
 c_2 & c_1h(2)&h(3)\\
 c_3 & c_2h(2)&c_1h(3)
\end{pmatrix}.
\]
Thus the entries are positive on and below the first superdiagonal and vanish strictly above it. This is the sign pattern which later implies primitivity.

We write $A\geq 0$ if all entries of the real square matrix $A$ are
nonnegative. Such a matrix is called primitive if there exists an integer
$k\geq 1$ such that $A^k>0$.

\begin{theorem}
\label{th:primitive}
Let \(n\geq 2\). The matrix \(M_n^h\) is primitive. Consequently, its
spectral radius \(\rho_n^h>0\) is an algebraically simple eigenvalue, and every
other eigenvalue \(\beta \) of \(M_n^h\) satisfies
\[
|\beta |<\rho_n^h.
\]
Moreover, there exist positive right and left Perron vectors
\(x,y\in\mathbb R^{n-1}\), normalized by
\[
\sum_{k=1}^{n-1}x_k=1,
\qquad
 y^T x=1,
\]
such that
\[
M_n^h x=\rho_n^h x,
\qquad
 y^T M_n^h=\rho_n^h y^T,
\]
and
\[
\lim_{m\to\infty}
\left(\frac{M_n^h}{\rho_n^h}\right)^m
=
xy^T.
\]
\end{theorem}

%%%%%%%%%%%%%%%%%%%%%%%%%%%%%%%%%%%%%%%
The proof is given in Section 
\ref{th:frob} 
%\ref{sec:proof-main} 
after the determinant
representation and the primitivity of $M_n^h$ have been established.
The special structure of the matrices \(M_n^h\) permits a comparison of the spectral radii as \(n\) varies. This yields the following strict monotonicity result, which is not a consequence of Perron--Frobenius theory alone.
%%%%%%%%%%%%%%%%%%%%%%%%%%%%%%%%%%%%%%%%%%%%%%%%%%%%%%%%%%%%%%%%%%%%
\begin{theorem}[Monotonicity of the spectral radius] \label{th:Mono}
For any sequence of positive numbers 
$\left( h(n)\right) _{n \geq 1} $ with $h(1)=1$ we obtain the strict inequalities
\begin{equation*}
0 < \rho_2^h < \rho_3^h < \rho_4^h < \ldots .
\end{equation*}
\end{theorem}
Let $\{ h(n)\,: n \in \N\}$ be bounded.
It was proved in \cite{HN24} that then the
set of all zeros of all $\pol _n^h(z)$ is bounded. Therefore, we obtain:

\begin{corollary}
Assume that $h=(h(n))_{n\ge1}$ is bounded and positive, with $h(1)=1$. Then the increasing sequence $(\rho_n^h)_{n\ge2}$ has a finite limit. More precisely,
\[
0<\rho_2^h<\rho_3^h<\ldots,
\qquad
\lim_{n\to\infty}\rho_n^h\le 9.7225\,\,\sup_{m\ge1}h(m).
\]
\end{corollary}

The paper is organized as follows. Section~\ref{a:hessenb}
proves the Hessenberg determinant representation for the generalized D'Arcais polynomials. Section~\ref{a:hessbew}
transforms this determinant representation into a finite-dimensional eigenvalue problem, introduces the matrix $M_n^h$, and reduces the required nonnegativity to the positivity of the coefficients $c_m$. Section~\ref{a:pf}
recalls the Perron--Frobenius theory for primitive nonnegative matrices and the comparison principle used later. Section~\ref{th:frob}
proves that $M_n^h$ is primitive. Section~\ref{sec:proof-main}
derives the dominant-zero theorem for %the
generalized D'Arcais polynomials, and %Section~8
applies it to the Nekrasov--Okounkov polynomials. Section~\ref{a:mono}
proves the strict monotonicity of the spectral radii $\rho_n^h$. Section~\ref{a:herausf}
records challenges and possible further directions, while Appendix~\ref{a:positiv}
gives a self-contained analytic proof of $c_m>0$. Appendix~\ref{app:ono-challenge3} presents Ken Ono's combinatorial solution to Challenge~\ref{cha:structural-positivity}, based on colored row-marked rectangles and an explicit injection.

%%%%%%%%%%%%%%%%%%%%%%%%%%%%%%%%%%%%%%%%% revise 6. Mai 2026 %%%%%%%%%%%%%%

\section{\label{a:hessenb}A Hessenberg Determinant Representation}
Let $H_n:= \func{diag}(h(1),h(2), \ldots,h(n))$ and $D_n= \func{diag}(1,-1, \ldots, (-1)^{n-1})$ be $n \times n$ matrices. Moreover, $I_n:= \func{diag}(1,1, \ldots,1)$.
For
\[
L_n:=\left( \left( L_{n}\right) %\ell
_{ij}\right)_{1 \leq i,j \leq n},
\qquad
%\ell
\left( L_{n}\right) _{ij}=
\begin{cases}
\sigma(i-j+1), & i\geq j,\\
0, & i<j,
\end{cases}
\]
its explicit lower-triangular Toeplitz form is
\[
L_n=
\begin{pmatrix}
\sigma(1) & 0 & 0 & \cdots & 0 \\
\sigma(2) & \sigma(1) & 0 & \cdots & 0 \\
\sigma(3) & \sigma(2) & \sigma(1) & \ddots & \vdots \\
\vdots & \vdots & \ddots & \ddots & 0 \\
\sigma(n) & \sigma(n-1) & \cdots & \sigma(2) & \sigma(1)
\end{pmatrix}.
\]
The matrix with entries equal to $1$ at positions $(i,i+1)$ and $0$ otherwise is the upper shift matrix (equivalently, the nilpotent Jordan block with eigenvalue zero). We have
%%%%%%%%%%%%%%%%%%%%%%%%%%%%%%%%%%%%%%%%%%%%
\[
(U_n)_{ij}
=
\begin{cases}
1, & j=i+1,\\
0, & \text{otherwise}.
\end{cases}
\]

%%%%%%%%%%%%%%%%%%%%%%%%%%%%%%%%%%%%%%%%%%%%
This gives the following Hessenberg determinant
representation.
\begin{proposition} We have
\begin{equation*} \label{Hessenberg}
\pol _n^{h}(z) =\det \left( z H_n^{-1} L_n - U_n \right).
\end{equation*}
\end{proposition}

\begin{proof}
The proof is by mathematical induction on $n$.
The case $n=1$ is obvious as $\pol _{1}^{h}\left( z\right) =z$.
Suppose now $n\geq 2$.
We use
the induction hypothesis that
$\pol _{j}^{h}\left( z\right) =\det \left( zH_{j}^{-1}L_{j}-U_{j}\right) $
for all $j<n$.

The matrix $zH_{n}^{-1}L_{n}-U_{n}$ has the following form as determinant
\[
\left|
\begin{array}{ccccc}
zh\left( 1\right) ^{-1}\sigma \left( 1\right) &-1
&0
&\cdots &0\\
zh\left( 2\right) ^{-1}\sigma \left( 2\right)
&zh\left( 2\right) ^{-1}\sigma \left( 1\right) &-1
&\ddots &0\\
\vdots &\vdots
&\ddots &\ddots &\vdots \\
zh\left( n-1\right) ^{-1}\sigma \left( n-1\right)
&zh\left( n-1\right) ^{-1}\sigma \left( n-2\right) &\cdots
&zh\left( n-1\right) ^{-1}\sigma \left( 1\right) &-1\\
zh\left( n\right) ^{-1}\sigma \left( n\right)
&zh\left( n\right) ^{-1}\sigma \left( n-1\right) &\cdots
&zh\left( n\right) ^{-1}\sigma \left( 2\right)
&zh\left( n\right) ^{-1}\sigma \left( 1\right)
\end{array}
\right| .
\]
We expand this with respect to the last row. So we have to compute the
determinant of the matrices obtained by
deleting
the last row and
the $k$th
column. This results in a matrix determinant
\[
\left|
\begin{array}{ccccc}
zh\left( 1\right) ^{-1}\sigma \left( 1\right) &-1
&0&\cdots
&0\\
zh\left( 2\right) ^{-1}\sigma \left( 2\right)
&zh\left( 2\right) ^{-1}\sigma \left( 1\right)
&-1&\ddots
&0\\
\vdots &\vdots
&\ddots
&\ddots &\vdots \\
zh\left( n-2\right) ^{-1}\sigma \left( n-2\right)
&zh\left( n-2\right) ^{-1}\sigma \left( n-3\right) &\cdots
&-1&0\\
zh\left( n-1\right) ^{-1}\sigma \left( n-1\right)
&zh\left( n-1\right) ^{-1}\sigma \left( n-2\right) &\cdots
&zh\left( n-1\right) ^{-1}\sigma \left( 1\right) &-1
\end{array}
\right| .
\]
Since we deleted column $k$
the
$\left( k-1\right) \times \left( n-k\right) $ block in the upper right corner
contains only zeros, i.~e.\ in rows $1$ to $k-1$ and (original) columns $k+1$ to
$n$. Therefore, we can compute the determinant as the product of the upper
left block of size $\left( k-1\right) \times \left( k-1\right) $ and the
lower right block of size $\left( n-k\right) \times \left( n-k\right) $.
Note that this lower right
block is a lower triangular matrix with $-1$ on the
diagonal so the determinant is $\left( -1\right) ^{n-k}$. The upper left
block is exactly the matrix $zH_{k-1}^{-1}L_{k-1}-U_{k-1}$. Therefore, by
expansion of the determinant
\begin{eqnarray*}
&&\det \left( zH_{n}^{-1}L_{n}-U_{n}\right) \\
&=&%\left( -1\right) ^{n+1}
\frac{z}{h\left( n\right)
}\sigma \left( n\right)
%\left( -1\right) ^{n-1}%\\
%%&&{}
+\sum _{k=2}^{n}
%%\left( -1\right) ^{n+k}
\frac{z}{h\left( n\right)
}\sigma \left(
n+1-k\right) \det \left( zH_{k-1}^{-1}L_{k-1}-U_{k-1}\right) 
%%\left( -1\right)^{n-k}
\\
&=&\frac{z}{h\left( n\right)
}\sigma \left( n\right) +\sum _{k=2}^{n}\frac{z}{h\left(
n\right)
}\sigma \left( n+1-k\right) \pol _{k-1}^{h}\left( z\right) \\
&=&\frac{z}{h\left( n\right) }\left( \sigma \left( n\right) +\sum _{j
=1}^{n-1}\sigma \left( j
\right) \pol _{n-j
}^{h}\left( z\right) \right)
=
\pol _{n}^{h}\left( z\right) .
\end{eqnarray*}
\end{proof}
%%%%%%%%%%%%%%%%%%%%%%%%%%%%%%%%%%%%%%%%%%%%%%%%%%%%%%%%%%%%%%%%%%%%%%%

\section{\label{a:hessbew}Proof of Theorem \ref{th:Hessenberg}}
The preceding determinant identity leads to the following eigenvalue representation.
\begin{equation}
\pol _n^h(z)=\det\left(H_n\right)^{-1}  \,\,\det \left(z I_n - A_n^h \right), \text{ where } A_n^h = L_n^{-1} H_n U_n.  \label{det:formula}
\end{equation}
\subsection{The inverse of $L_n$}
To proceed we determine the entries of $L_n^{-1}$. Therefore, the sequence $\left(
c_m\right)
_{m \geq 0}$ defined in (\ref{cm}) is also given by the coefficients
of the reciprocal of $$S\left( t\right) :=\sum _{m
=0}^{\infty }\sigma \left( m +1 \right) \left( -t\right) ^{m}.$$
Note this power series is regular at $t=0$.
\begin{equation*}
\sum_{m=0}^{\infty} c_m \, t^m = \frac{1}{S(t)}.
\end{equation*}

\begin{proposition}
We obtain
\[
L_n^{-1}=
\begin{pmatrix}
1 & 0 & 0 & \cdots & 0 \\
-c_1 & 1 & 0 & \cdots & 0 \\
c_2 & -c_1 & 1 & \ddots & \vdots \\
\vdots & \vdots & \ddots & \ddots & 0 \\
(-1)^{n-1}c_{n-1} & (-1)^{n-2}c_{n-2} & \cdots & -c_1 & 1
\end{pmatrix}.
\]
\end{proposition}
%%%%%%%%%%%%%%%%%%%%%%%%%%%%%%%%%%%%%%%%%%%%%%%%%%%%%%%%%%%%%%

%%%%%%%%%%%%%%%%%%%%%%%%%%%%%%%%%%%%%%%%%%%%%%%%%%%%%%%%%%%%%%
\begin{proof}
%%From (\ref{cm}) follows
We have
\[
1
=
\left( \sum _{m=0}^{\infty }c_{m}t^{m}\right)
\left( \sum _{m=0}^{\infty }\sigma \left( m+1\right) \left( -t\right) ^{m}\right)
=
\sum _{k=0}^{\infty }\sum _{m=0}^{k}\left( -1\right) ^{k-m}\sigma \left(
k-m+1\right) c_{m}t^{k}
.
\]
By comparison of coefficients we obtain for $k\geq 1$ that
\[
0=
\sum _{m=0}^{k}\sigma \left( k-m+1\right)
\left( -1\right) ^{k-m}c_{m}
.
\]
We define a
matrix via the coefficients
$$\left( \left( -1\right) ^{k
-m}c_{k-m}\right) _{k,m=1,2,\ldots ,n}.$$
where $c_{j}=0$ if $j\leq -1$. If we now compute the product of $L_{n}$ with
this matrix we obtain as entry in the $i$th row and $m$th column when $i>m$
that
$$\sum _{k=m}^{i}\sigma \left( i-k+1\right) \left( -1\right) ^{k
-m}c_{k-m}
=\sum _{k=0}^{i-m}\sigma \left( i-m-k+1\right) \left( -1\right) ^{k
}c_{k
}=0.$$
If $i=m$ this results in $1$ and for $i<m$ it is $0$. Therefore, it is the
inverse matrix. 
\end{proof}
\subsection{Positivity result}
\begin{proposition}
\label{prop:pos}The sequence $\left(
c_m\right)
_m$ 
satisfies $c_m >0$ for $m \geq 0$.
\end{proposition}

Gandhi \cite{Ga69} indicated a combinatorial argument for the positivity of the coefficients $c_m$
defined in Proposition \ref{prop:pos}.
However, the argument is only sketched there and does not appear to contain all details needed for the present application. 
We give a self-contained analytic proof of the positivity of the coefficients $c_m$. The argument combines elementary one-variable complex analysis with a meromorphic singularity analysis of the reciprocal generating function $1/S(t)$. The complex-analytic tools used below are 
Rouch\'e's theorem, residues, and Cauchy's coefficient formula; see Remmert \cite{Re91} for background. For the generating-function viewpoint, especially the principle that the nearest singularity controls the coefficients of a meromorphic generating function, see Flajolet and Sedgewick \cite{FS09}. For the convenience of the reader
we give an overview of the analytic proof and put the details in Appendix~\ref{a:positiv}. An independent combinatorial solution to Challenge~\ref{cha:structural-positivity}, due to Ken Ono, is included as Appendix~\ref{app:ono-challenge3}.

\subsubsection{Short summary of the proof strategy}
The proof starts from the reciprocal generating function
\[
  \sum_{m\geq 0} c_m t^m = \frac{1}{S(t)},
  \qquad
  S(t)=\sum_{m\geq 0} \sigma(m+1)(-t)^m .
\]
The aim is to prove that all coefficients $c_m$ are positive. The main steps are as follows.

\begin{enumerate}[label=\textup{(\arabic*)}, leftmargin=2.2em]
  \item First, the infinite series $S(t)$ is approximated by a finite truncation $S_{\zerlegung }\left( t\right)
  $, with explicit error estimates for $S(t)-S_{\zerlegung }\left( t\right)
  $ and its derivative on the disk $|t|\leq 1/2$.

  \item These estimates are used together with a finite computation on the circle $|t|=1/2$ to prove that $S(t)$ does not become too small on the boundary. Rouche's theorem then shows that $S(t)$ and the chosen truncation have the same number of zeros in $|t|<1/2$.

  \item One obtains a unique zero $t_0$ of $S(t)$ in the disk $|t|<1/2$. The proof also gives explicit bounds showing that $t_0$ is real, lies in a short positive interval, and is simple. Consequently $1/S(t)$ has a simple pole at $t_0$.

  \item The reciprocal is then decomposed as
  \[
    \frac{1}{S(t)}=\frac{\residuum }{t-t_0}+\hilfsftn (t),
    \qquad
    \residuum =\func{Res}{}_{t=t_0}\frac{1}{S(t)}=\frac{1}{S'(t_0)},
  \]
  where $\hilfsftn (t)$ is holomorphic in a disk larger than $
t_{0}$. Since $S'(t_0)<0$, one has $\residuum <0$.

  \item Comparing coefficients yields
  \[
    c_m = -\residuum \, t_0^{-m-1} + \omega _m,
  \]
  where $\omega _m$, $m\geq 0$, are
  the coefficients of $\hilfsftn (t)$. Cauchy's estimate bounds $\omega _m$, while the pole contribution $-\residuum \,t_0^{-m-1}$ is positive and dominates for all sufficiently large $m$.

  \item The remaining finitely many coefficients are checked directly. This proves $c_m>0$ for every $m\geq 0$.
\end{enumerate}

\subsection{The matrix $M_n^h$}

A modification of the formula (\ref{det:formula})
leads to
\begin{equation*} %\label{det:formula}
\pol _n^h(z) =
\det\left(H_n\right)^{-1}  \,\,\det \left(z I_n +%- (-1)
B_n^h  \right),
\end{equation*}
where
\begin{equation*}
B_n^{h}:= - D_n \, A_n^{h} D_n
\end{equation*}
and
\[
B_n^{h}=
\begin{pmatrix}
0 & 1 & 0 & 0 & \cdots & 0\\
0 & c_1 & h(2) & 0 & \cdots & 0\\
0 & c_2 & h(2)c_1 & h(3) & \ddots & \vdots \\
0 & c_3 & h(2)c_2 & h(3)c_1 & \ddots & 0 \\
\vdots & \vdots & \vdots & \vdots & \ddots & h(n-1)\\
0 & c_{n-1} & h(2)c_{n-2} & h(3)c_{n-3} & \cdots & h(n-1)\, c_1
\end{pmatrix}.
\]
Note that the matrix $B_n^{h}$ is not primitive. Nevertheless, we directly obtain:
\begin{equation*}
\label{identity}
\pol _n^h(-z) =
\det\left(H_n\right)^{-1}  \,\, (-1)^n z \det \left(z I_{n-1} - 
M_n^h \right),
\end{equation*}
where
\begin{equation} \label{Mn}
M_n^h=
\begin{pmatrix}
c_1 & h(2) & 0 & \cdots & 0\\
c_2 & h(2)c_1 & h(3) & \cdots & 0\\
c_3 & h(2)c_2 & h(3)c_1 & \ddots & \vdots\\
\vdots & \vdots & \vdots & \ddots & h(n-1)\\
c_{n-1} & h(2)c_{n-2} & h(3)c_{n-3} & \cdots & h(n-1)c_1
\end{pmatrix}.
\end{equation}
The matrix 
$M_n^h$
has nonnegative entries.

Therefore, Theorem \ref{th:Hessenberg} is proven.
Before we proceed we recall some results from the Perron--Frobenius theory.
%%%%%%%%%%%%%%%%%%%%%%%%%%%%%%%%%%%%%%%%%%%%%%%%%%%%%%%%%%%%%%%%%%%%%%
%%%%%%%%%%%%%%%%%%%%%%%%%%%%%%%%%%%%%%%%%%%%%%%%%%%%%%%%%%%%%%%%%%%%%%

\section{\label{a:pf}Perron--Frobenius Theory} %%%%%%stop

The following result is due to Perron and Frobenius.
\begin{theorem}\label{th:PF}
Let $ A \in \mathbb{R}^{d \times d}$ be primitive and let \(\rho=\rho(A)\) denote its
spectral radius.  Then \(\rho>0\), \(\rho\) is an algebraically simple
eigenvalue of \(A\), and every other eigenvalue \(\beta \) of \(A\) satisfies
\(|\beta |<\rho\).  Moreover, there are uniquely determined vectors
\(x,y>0\), normalized by
\[
   \sum_{k=1}^{d}x_k=1, \qquad y^T x=1,
\]
such that
\[
   Ax=\rho x, \qquad y^T A=\rho y^T.
\]
With this normalization,
\[
   \lim_{m\to\infty} \rho^{-m}A^m=xy^T.
\]
\end{theorem}
Perron's original theorem treated the case of positive matrices.  The
extension to nonnegative matrices, and in particular to primitive matrices,
is due to Frobenius; see, for example,
Seneta \cite[Theorem~1.1]{Se06} and
Horn--Johnson \cite[Section~8.5]{HJ12}.

%%%%%%%%%%%%%%%%%%%%%%%%%%%%%%%%%%%%%%%%%%%%%%%%%%%%%%%%%%%%%%%%%%%%%%%%%%%%%%%
\begin{lemma}[Strict Perron--Frobenius comparison] \label{comparison}
Let $A$ be primitive. Let $ y \in \mathbb{R}^n$ be
positive and $\beta \in \mathbb{R}$. Suppose 
\begin{equation*}
A \, y \geq \beta \, y
\end{equation*}
and at least one component of $Ay-\beta \, y$ is strictly positive. Let $\rho %_
\left( A\right) $ be the spectral radius of $A$, then we have
\begin{equation*}
\rho %_
\left( A \right) >\beta.
\end{equation*}
\end{lemma}
%%%%%%%%%%%%%%%%%%%%%%%%%%%%%%%%%%%%%%%%%%%%%%%%%%%%%%%%%%%%%%%%%%%%%%%%%%%%%%%
\begin{proof}
By Perron--Frobenius theory there exists a vector $x>0$ such that
$A^Tx=\rho %_
\left( A\right) x$. Therefore
\[
\rho %_
\left( A\right) x^Ty=x^TAy.
\]
Since $Ay\ge \beta y$ and $Ay-\beta y$ has at least one positive
component, while $x>0$, we obtain
\[
x^TAy > \beta x^Ty.
\]
Thus
\[
\rho %_
\left( A\right) x^Ty > \beta x^Ty.
\]
Since $y>0$ and $x>0$, one has $x^Ty>0$, and hence $\rho
\left( A\right) >\beta$.

\end{proof}

%%%%%%%%%%%%%%%%%%%%%%%%%%%%%%%
%%%%%%%%%%%%%%%%%%%%%%%%%%%%%%%

%%%%%%%%%%%%%%%%%%%%%%%%%%%%%%%
%%%%%%%%%%%%%%%%%%%%%%%%%%%%%%%
%%%%%%%%%%%%%%%%%%%%%%%%%%%%%%%
\section{Proof of Theorem \ref{th:primitive}}
\label{th:frob}
From the entry formula (\ref{Mn}), $M_n^h$ is the leading principal submatrix of $M_{n+1}^h$. More precisely, let $n \geq 2$. Then
\begin{equation*} \label{decom}
 M_{n+1}^h=\begin{pmatrix}
 M_n^h &u_n\\
 v_n^T&c_1h(n)
 \end{pmatrix},
\end{equation*}
where
\[
 u_n=h(n)e_{n-1} \in \mathbb{R}^{n-1}, \,\, e_{n-1}=(0,\ldots,0,1)\in \mathbb{R}^{n-1},
\]
and
\begin{equation*}
 v_n^{T}
 =\bigl(c_nh(1),c_{n-1}h(2),\ldots,c_2h(n-1)\bigr) \in \mathbb{R}^{n-1}.
\end{equation*}
%Let $d_n=c_1 h(n)$.
Then we have
\begin{equation*} \label{decom vektor}
M_n^h \geq 0, \, u_n \geq 0, \, u_n \neq 0, \, v_n >0 \text{ and } %d_n
c_{1}h\left( n\right) >0.
\end{equation*}
%%%%%%%%%%%%%%%%%%%%%%%%%%%%%%%%%%%%%%%%% insert
\begin{lemma}\label{lem:primitive-Mnh}
Let $n\ge2$. Let $\left(
h(n)\right)
_{n\geq 1}$ be a normalized sequence of positive
real numbers. Further, by Proposition \ref{prop:pos}, $c_m>0$ for all
$m\ge0$. Then $M_n^h$ is primitive. For $n=2$ already $M_2^h>0$, and for
$n\ge3$ one has
\[
   (M_n^h)^{n-2}>0.
\]
Moreover, for $n\ge3$, the exponent $n-2$ is sharp.
\end{lemma}

\begin{proof}
By the explicit formula for $M_n^h$, together with $c_m>0$ and $h(j)>0$,
we have
\[
   (M_n^h)_{ij}>0 \quad\text{if } j-i\le 1,
   \qquad
   (M_n^h)_{ij}=0 \quad\text{if } j-i>1 .
\]
Thus $M_n^h$ is positive on and below
the first
superdiagonal, and zero strictly above the first superdiagonal.

We claim that, for every $\Mnhexp
\ge1$,
\[
   \bigl((M_n^h)^{\Mnhexp }
\bigr)_{ij}>0
   \quad\text{if } j-i\le \Mnhexp ,
   \qquad
   \bigl((M_n^h)^{\Mnhexp }
\bigr)_{ij}=0
   \quad\text{if } j-i>\Mnhexp .
\]
The claim follows by induction on $\Mnhexp $. The case $\Mnhexp
=1$ is exactly the sign
pattern above. Suppose the assertion is true for some $\Mnhexp \ge1$. Then
\[
   \bigl((M_n^h)^{\Mnhexp
+1}\bigr)_{ij}
   =
   \sum_{\ell=1}^{n-1}
   \bigl((M_n^h)^{\Mnhexp }\bigr)_{i\ell}(M_n^h)_{\ell j}.
\]
A nonzero summand can occur only if
\[
   \ell-i\le \Mnhexp
,
   \qquad
   j-\ell\le 1.
\]
Hence a nonzero summand can occur only when $j-i\le \Mnhexp
+1$, which proves the
vanishing for $j-i>\Mnhexp
+1$. Conversely, if $j-i\le \Mnhexp
+1$, choose
$\ell=\max\{1,j-1\}$. Then $1\le\ell\le n-1$, $\ell-i\le \Mnhexp
$, and
$j-\ell\le1$. Therefore,
the corresponding summand is strictly positive,
and all summands are nonnegative. Hence
\[
   \bigl((M_n^h)^{\Mnhexp
+1}\bigr)_{ij}>0.
\]
This proves the induction.

The matrix $M_n^h$ has size $n-1$, so the largest possible value of
$j-i$ is $n-2$. Therefore, for $n\ge3$,
\[
   (M_n^h)^{n-2}>0.
\]
For $n=2$, the matrix $M_2^h$ is a positive $1\times1$ matrix. Thus
$M_n^h$ is primitive for every $n\ge2$.

Finally, if $n\ge3$ and $\Mnhexp
<n-2$, then the $(1,n-1)$-entry of
$(M_n^h)^{\Mnhexp }$ is zero by the sign pattern just proved, since
$(n-1)-1=n-2>\Mnhexp
$. Hence the exponent $n-2$ is sharp.
\end{proof}
\subsection{Final steps in the proof of Theorem \ref{th:primitive}}

By Lemma~\ref{lem:primitive-Mnh}, the matrix $M_n^h$ is primitive. The remaining
assertions follow directly from the Perron--Frobenius theorem for primitive
nonnegative matrices, stated as Theorem~\ref{th:PF}.

\section{Proofs of the
Dominant-Zero
Results}
\label{sec:proof-main}

\begin{proof}[Proof of Theorem~\ref{th:Darcais}]
By Theorem~\ref{th:Hessenberg}, the nonzero zeros of $\pol _n^h(-z)$ are precisely
the eigenvalues of $M_n^h$, counted with algebraic multiplicity. Since $M_n^h$
is primitive and nonnegative, Perron--Frobenius theory implies that there is a
unique eigenvalue on the spectral circle, namely $\rho_n^h=\rho(M_n^h)>0$, and
that this eigenvalue is algebraically simple. Hence $\pol _n^h(-z)$ has a unique
zero of maximal modulus, namely $z=\rho_n^h$. Equivalently, $\pol _n^h(z)$ has the
unique dominant zero $-\rho_n^h$, which is real, negative, and simple. The zero
at the origin does not affect dominance, since $\rho_n^h>0$.
\end{proof}

%\section
\begin{proof}[Proof of Corollary \ref{NO}]
The D'Arcais polynomials provide the Nekrasov--Okounkov polynomials \cite{HN19}
by shifting
$z \mapsto z+1$:
\begin{equation*}
\sum_{n=0}^{\infty}\pol _n^{h_1}(z) \,q^n =
\prod_{n=1}^{\infty}
\left(  1 - q^n \right)^{-z}.
\end{equation*}
Although the shift from \(\pol _n^{h_1}\) to \(\nop _n\) changes the coefficients, it preserves the dominant-zero property in the present setting, by the triangle inequality.
Therefore,
\[
\nop _n(z)=\pol ^{h_1}_n(z+1).
\]
By Theorem~\ref{th:Darcais}, the polynomial $\pol _n^{h_1}(z)$ has a unique
zero $\alpha_n$ of maximal modulus. This zero is real, negative, and simple.
Write $\alpha_n=-\negmaxmod
_{n}$, with $\negmaxmod
_{n}>0$. If $\beta\neq \alpha_n$ is any other zero of $\pol ^{h_1}_n(z)$, then $
\left| \beta \right| <\negmaxmod _{n}$.
The zeros of $\nop _n(z)$ are precisely $\gamma=\beta-1$, where $\beta$ ranges over the zeros of $\pol ^{h_1}_n(z)$. The zero corresponding to $\alpha_n$ is
\[
\alpha_n-1=-\negmaxmod
_n-1,
\]
which is real and negative. Its modulus is $\negmaxmod
_n+1$. For every other zero $\beta-1$, the triangle inequality gives
\[
|\beta-1|\leq |\beta|+1<\negmaxmod
_{n}+1=|\alpha_n-1|.
\]
Hence $\nop _n(z)$ has a unique zero of maximum
modulus. It is real, negative, and simple. Equivalently, $\nop _n(-z)$ has the unique positive zero $
\negmaxmod _{n}+1$ of maximum
modulus. This proves the Corollary.
\end{proof}

\section{\label{a:mono}Proof of Theorem \ref{th:Mono}}
We apply
Lemma \ref{comparison}. We have that $M_{n}^{h}
\geq 0$ and primitive for all $n \geq 2
$. Therefore we can apply 
Theorem \ref{th:PF}.
There exist $x>0$ in $\mathbb{R}^{n-1}$ such that
\begin{equation*}
M_{n}^{h}x=\rho _{n}^{h} \, x,
\end{equation*}
where $\rho _{n}^{h}$ is the spectral radius of $M_{n}^{h}$.
We extend $x
$ by one positive valued component to
$
{x}_{\xi
}=\left(
\begin{array}{c}
x
\\
\xi
\end{array}
\right) $. 
%%%%%%%%%%%%%%%%%%%%%%%%%%%%%%%%%%%%%%%%
Then $$M_{n+1}^{h} \,
{x}_{\xi
}=\left(
\begin{array}{c}
M_{n}^{h}x
+\xi
u_{n} \\
v_{n}^{T}x
+\xi
c_{1}h\left( n\right)
\end{array}
\right).$$ Since $u_{n}\geq 0$ we have $M_{n}^{h}x
+\xi
u_{n}\geq \rho _{n}^{h} x
$.
If $\rho _{n}^{h} \leq c_{1}h\left( n\right) $ automatically
$$v_{n}^{T}x
+\xi
c_{1}h\left( n\right)
\geq \rho _{n}^{h}\xi
.$$
For $0<c_{1}h\left( n\right) <\rho _{n}^{h}$
we choose
$$0<\xi
<
\frac{v_{n}^{T} \, x}{\rho _{n}^{h}
-c_{1}h\left( n\right) }.$$
Since \(M_{n+1}^h x_{\xi
}\geq \rho_n^h x_{\xi
}\) and the inequality is
strict in at least one component, Lemma~5.2 gives
\[
\rho(M_{n+1}^h)>\rho_n^h.
\]
Thus \(\rho_{n+1}^h>\rho_n^h\), as claimed.

%%%%%%%%%%%%%%%%%%%%%%%%%%%%%%%%%%%%%%%%%%%%%%%%%%%%%%%%%
%%%%%%%%%%%%%%%%%%%%%%%%%%%%%%%%%%%%% Challenge
%%\section{Challenges}

%%%%%%%%%%%%%%%%%%%%%%%%%%%%%%%%%%%%%%%%%%%% Future
\section{\label{a:herausf}Future Perspectives and Challenges}

The results of this paper give an exact finite-dimensional Perron--Frobenius approach for the dominant zero of generalized D'Arcais polynomials and of the Nekrasov--Okounkov polynomials. Since the nonzero zeros are realized as the spectrum of the explicit primitive nonnegative Hessenberg matrix $M_n^h$, several natural questions about zero distributions become concrete spectral and structural questions about this matrix family.

\subsection{Simplicity beyond the Perron zero}

Theorem~2.1 proves that the dominant zero is real, negative, and simple. The natural next problem is whether this simplicity phenomenon extends to all zeros.

\begin{challenge}[Simplicity]
For $n\geq 1$, decide whether all zeros of the Nekrasov--Okounkov polynomial $\nop _n(z)$ are simple; more generally, determine for which positive sequences $h=(h(n))_{n\geq 1}$ all zeros of $\pol _n^h(z)$ are simple.
\end{challenge}

This challenge is close to the conjectural picture around the Nekrasov--Okounkov formula and D'Arcais polynomials. Amdeberhan's problem list gave an important starting point \cite{Am15}; in earlier work we refined related questions on simplicity and zero location \cite{HN19}. Partial evidence is provided by the known simplicity of
 zeros of $\pol _{n}\left( z
 \right) $ for $n$ a prime \cite{HN26} and integral zeros for
 $n=p^m$ or $n=p^m+1$ for all $m \in \mathbb{N}_0$ \cite{HN18}. Recent work on hook polynomials and partition statistics provides further context \cite{BCFR25}.

\subsection{Spectral gap and non-Perron zeros}

Let
\[
\mu_n^h=\max\{ |\beta|:\beta\in\operatorname{spec}(M_n^h),\ \beta\neq \rho_n^h\},\qquad
\spektral _n^h=1-\frac{\mu_n^h}{\rho_n^h}.
\]
By primitivity, $\spektral _n^h>0$ for each fixed $n$. The difficult question is whether this separation can be estimated uniformly or asymptotically.

\begin{challenge}[Spectral gap]
For natural choices of $h$, especially $h_s(n)=n^s$, determine the asymptotic behavior of the relative gap $\spektral _n^h$, and decide whether $\spektral _n^h\to 0$ as $n\to\infty$.
\end{challenge}

A quantitative gap would measure how strongly the Perron zero dominates the remaining zeros and would give a rate of convergence in the Perron projection
\[
\left(\frac{M_n^h}{\rho_n^h}\right) ^{\Mnhexp }
\longrightarrow xy^T .
\]
The matrices here are non-symmetric and should not be viewed as adjacency matrices of regular graphs. Nevertheless, the analogy with spectral-gap problems is useful at the conceptual level; one standard reference is \cite{DSV03}.

\subsection{Challenge 3: Structural positivity}

The construction depends on the positivity of the coefficients $c_m$ defined by
\[
\sum_{m\geq0}c_m t^m=
\left(\sum_{m\geq0}\sigma(m+1)(-t)^m\right)^{-1}.
\]
This positivity makes the Hessenberg matrix $M_n^h$ nonnegative. The analytic proof in Appendix~\ref{a:positiv} is retained because its method may be useful for related divisor-sum variants, for instance with $\sigma_2(n)=\sum_{d\mid n}d^2$. It is also natural to isolate the following conceptual form of the problem.

\begin{challenge}[Structural positivity]\label{cha:structural-positivity}
Find a conceptual proof of $c_m>0$ for all $m\geq0$, preferably by a combinatorial model, a sign-reversing injection, or a positivity-preserving recurrence.
\end{challenge}

Challenge~\ref{cha:structural-positivity} has been solved by Ken Ono. His solution, included as Appendix~\ref{app:ono-challenge3}, gives a short combinatorial proof based on colored row-marked rectangles. The key point is to interpret suitable auxiliary coefficients $\coefauxprod _n$ as signed counts and then to cancel the negative objects by an explicit injection. Gandhi indicated a combinatorial approach to related coefficients \cite{Ga69}, while Andrews' classical account places such positivity questions naturally within partition theory \cite{An76}. The perspective of Naskr\k{e}cki--Ono on AI-assisted mathematical discovery is relevant as methodological context \cite{NO25AI};
Appendix~\ref{a:positiv}, however, gives a genuine proof rather than numerical evidence.

%%%%%%%%%%%%%%%%%%%%%%%%%%%%%%%%%%%%%%%%%%%%%%%%%%%%%%
\appendix
\section{\label{a:positiv}Proof of Proposition~\ref{prop:pos}}

We keep the notation introduced
ahead of Proposition \ref{prop:pos}. Thus
\[
        S(t)=\sum_{m=0}^{\infty}\sigma(m+1)(-t)^m,
        \qquad
        \sum_{m=0}^{\infty}c_m t^m=\frac{1}{S(t)} .
\]
For an integer $\zerlegung
\geq 1$ let
\[
        S_{\zerlegung }\left( t\right) :=\sum_{m=0}
        ^{\zerlegung
        -1}\sigma(m+1)(-t)^m
\]
be the $\zerlegung
$th truncation of $S$. We shall use throughout the elementary bound
\[
        \sigma(m+1)\leq (m+1)(m+2).
\]
For $0\leq \vartheta
<1$ put
\[
        \Psi _{\zerlegung }
        \left( \vartheta
        \right) :=
        \sum_{m=\zerlegung
        }^{\infty}(m+1)(m+2)\vartheta
        ^m
        =\frac{\mathrm{d}^2}{\mathrm{d}\vartheta
        ^2}\left(\frac{\vartheta
        ^{\zerlegung
        +2}}{1-\vartheta
        }\right).
\]
This equates to
\[
        \Psi _{\zerlegung }\left(
        \vartheta
        \right) =
        \frac{\left( \zerlegung +2\right)
        \left( \zerlegung +1\right) \vartheta
        ^{\zerlegung }-2\left( \zerlegung +2\right)
        \zerlegung \vartheta
        ^{\zerlegung +1}+\left( \zerlegung +1\right)
        \zerlegung \vartheta
        ^{\zerlegung +2}}
             {(1-\vartheta
             )^3}.
\]

\begin{lemma}\label{lem:tail-estimates}
For $|t|\leq \vartheta
<1$, one has
\begin{equation}\label{eq:tail-estimates}
        \left| S(t)-S_{\zerlegung }\left( t\right)
        \right| \leq \Psi _{\zerlegung }\left( \vartheta
        \right) ,
        \qquad
        \left| S^{\prime }\left( t\right) -S_{\zerlegung
        }^{\prime }\left( t\right) \right| \leq \Psi
        _{\zerlegung }^{\prime }\left( \vartheta
        \right) .
\end{equation}
In particular, for $|t|\leq 1/2$,
\begin{equation}\label{eq:tail-estimates-19}
        |S(t)-S_{19}(t)|\leq \Psi_{19}\!\left(\frac12\right)=\frac{29}{16384},
        \qquad
        |S'(t)-S_{19}'(t)|\leq \Psi_{19}'\!\left(\frac12\right)=\frac{2343}{32768}.
\end{equation}
\end{lemma}

\begin{proof}
The first estimate follows from
\[
 \left| S\left( t\right)
 -S_{\zerlegung }\left( t\right) \right|
 \leq \sum_{m=\zerlegung
 }^{\infty}\sigma(m+1)\vartheta
 ^m
 \leq \sum_{m=\zerlegung
 }^{\infty}(m+1)(m+2)\vartheta
 ^m=\Psi _{\zerlegung }\left( \vartheta
 \right) .
\]
Differentiating the majorizing term yields
an estimate for the derivative. The two displayed
values are obtained by substituting $\zerlegung =19$ and $\vartheta
=1/2$.
\end{proof}

\begin{lemma}\label{lem:derivative-bounds}
For every $\zerlegung
\geq 1$ and every $t\in\mathbb C$ with $|t|\leq 1/2$, one has
\begin{equation}\label{eq:derivative-bounds}
        \left| S_{\zerlegung }^{\prime }\left( t\right)
        \right| ,\ |S'(t)|\leq 96,
        \qquad
        \left| S_{\zerlegung }^{\prime \prime }
        \left( t\right) \right| ,\ |S''(t)|\leq 768.
\end{equation}
\end{lemma}

\begin{proof}
The estimate for the first derivatives follows from
\[
 \sum_{m=1}^{\infty}m(m+1)(m+2)\vartheta
 ^{m-1}=
 \frac{\mathrm{d}^{3}}{\mathrm{d}\vartheta
 ^{3}}\frac{1}{1-\vartheta
 },
\]
and the estimate for the second derivatives follows from
\[
 \sum_{m=2}^{\infty}m(m-1)(m+1)(m+2)\vartheta
 ^{m-2}=
 \frac{\mathrm{d}^{4}}{\mathrm{d}\vartheta
 ^{4}}\frac{1}{1-\vartheta
 }.
\]
The same bounds apply to the finite truncations.
\end{proof}

\begin{lemma}\label{lem:boundary}
On the circle $|t|=1/2$ one has
\begin{equation}\label{eq:boundary-S}
        \min_{|t|=1/2}|S(t)|>\frac{2}{55}.
\end{equation}
\end{lemma}

\begin{proof}
Consider the sampling points
\[
        t_j=\frac12\exp\left(\frac{\pi \mathrm{i} j}{24
0}\right),
        \qquad 1\leq j
\leq 480
.
\]
A
computation with PARI/GP gives
\[
        \min_{1
        \leq j
\leq 480
}|S_{19}(t_j)|=\frac{46197}{131072}.
\]
Let $t=\frac12 \mathrm{e}^{\mathrm{i}
\varphi }$. Choose $j\in \mathbb{Z}$
such that
\[
        \left|\varphi -\frac{\pi j}{24
0}\right|\leq \frac{\pi}{48
0}.
\]
Then
\[
        |t-t_j|\leq \frac12\left|\varphi -\frac{\pi j}{24
0}\right|
        \leq \frac{\pi}{96
0}.
\]
By \eqref{eq:derivative-bounds},
\[
        |S_{19}(t)|\geq |S_{19}(t_j)|-96|t-t_j|
        \geq \frac{46197}{131072}-\frac{\pi}{1
0}.
\]
Finally,
using $\pi<355/113$ and \eqref{eq:tail-estimates-19}, we obtain
\[
        |S(t)|\geq |S_{19}(t)|-|S(t)-S_{19}(t)|>
        \frac{46197}{131072}-\frac{355
}{1130
}-\frac{29}{16384}
        >
\frac{2
}{55
}
\]
which proves (\ref{eq:boundary-S}).
\end{proof}

\begin{lemma}\label{lem:small-zero}
The function $S$ has precisely one zero in $|t|<1/2$. This zero is simple,
real, and satisfies
\begin{equation}\label{eq:t0-interval}
        t_0\in \left(\frac{315}{767},\frac{586}{1423}\right).
\end{equation}
\end{lemma}

\begin{proof}
Put
\begin{equation}
        \links :=\frac{315}{767},
        \qquad
        \rechts :=\frac{586}{1423},
        \qquad
        \naeherung :=\frac{29}{16384}.
\label{eq:raender}
\end{equation}
Exact rational arithmetic gives
\[
        S_{19}(\links )>\naeherung ,
        \qquad
        S_{19}(\rechts )<-\naeherung .
\]
By \eqref{eq:tail-estimates-19}, this implies $S(\links )>0$ and $S(\rechts )<0$. Hence
$S$ has a real zero $t_0\in(\links ,\rechts )$.

On $|t|=1/2$, \eqref{eq:tail-estimates-19} and \eqref{eq:boundary-S} yield
\[
        |S(t)-S_{19}(t)|\leq \frac{29}{16384}<|S
        (t)|.
\]
By Rouch\'e's theorem, $S$ and $S_{19}$ have
the same number of zeros in $|t|<1/2$, counted with multiplicity. The
polynomial $S_{19}$ has exactly one zero in the disk $|t|<1/2$, counted
with multiplicity. This can be checked with PARI/GP.
Thus the real
zero $t_0$ found above is the unique zero in this disk and is simple.
\end{proof}

\begin{lemma}\label{lem:residue-bound}
We have
\begin{equation}\label{eq:residue-bound}
        -\frac49<\func{Res}{}_{t=t_0}\frac1{S(t)}
        <-\frac{10}{41}.
\end{equation}
\end{lemma}

\begin{proof}
Let $\links ,\rechts $ be as in (\ref{eq:raender}). By \eqref{eq:tail-estimates-19},
\eqref{eq:derivative-bounds}, and $t_0\in(\links ,\rechts )$, we have
\[
\begin{aligned}
S'(t_0)
&\leq S_{19}'(\rechts )+\frac{2343}{32768}+768(\rechts -\links )<-\frac94,\\
S'(t_0)
&\geq S_{19}'(\links )-\frac{2343}{32768}-768(\rechts -\links )>-\frac{41}{10}.
\end{aligned}
\]
Hence
\[
        -\frac{41}{10}<S'(t_0)<-\frac94.
\]
Since $\func{Res}_{t=t_0}\frac1{S(t)}=\frac1{S'(t_0)}$
taking reciprocals
gives \eqref{eq:residue-bound}.
\end{proof}

\begin{lemma}\label{lem:large-coefficients}
For every $m\geq 21
$, one has $c_m>0$.
\end{lemma}

\begin{proof}
The function $1/S(t)$ has, in $|t|<1/2$, exactly one pole, namely the simple pole
at $t_0$. Let
$\residuum
:=\func{Res}_{t=t_0}\frac1{S(t)}$.
Define
\[
        \hilfsftn (t):=\frac1{S(t)}-\frac{\residuum
        }{t-t_0}.
\]
Then $\hilfsftn $ is holomorphic in a neighbourhood of the closed disk $|t|\leq 1/2$.
Write
\[
        \hilfsftn (t)=\sum_{m=0}^{\infty}\omega _m t^m.
\]
For $|t|=1/2$, Lemmata~\ref{lem:boundary} and \ref{lem:residue-bound} give
\[
 |\hilfsftn (t)|\leq \frac1{|S(t)|}+\frac{|\residuum|}{|t-t_0|}
 <\frac{55}{2}
+\frac{4/9}{1/2-586/1423}
 =\frac{147013
}{4518
}.
\]
Therefore, Cauchy's estimate
yields
\begin{equation}\label{eq:h-m-bound}
        |\omega _m|\leq \frac{147013
}{4518
}\,2^m.
\end{equation}
On the other hand, for $|t|<t_0$,
\[
        \frac{\residuum}{t-t_0}
        =-\sum_{m=0}^{\infty}\residuum\,t_0^{-m-1}t^m.
\]
Comparing coefficients in $1/S(t)=\residuum/(t-t_0)+\hilfsftn (t)$ gives
\begin{equation*}\label{eq:coefficient-decomposition}
        c_m=-\residuum\,t_0^{-m-1}+\omega _m.
\end{equation*}
Using \eqref{eq:residue-bound}, \eqref{eq:t0-interval}, and
\eqref{eq:h-m-bound}, we obtain
\[
        c_m>\frac{10}{41}\left(\frac{586}{1423}\right)^{-m-1}
        -\frac{147013
}{4518
}\,2^m .
\]
For $m=21
$ the right-hand side is positive by exact rational arithmetic. Since
$2<
1423/586
$, the inequality remains true for all $m\geq 21
$.
\end{proof}

We now finish the proof of %the positivity p
Proposition \ref{prop:pos}. Since $c_0=1$, it remains,
by Lemma~\ref{lem:large-coefficients}, only to check $1\leq m\leq 20
$. These
values are listed in Table~\ref{tab:small-c}; all are positive.

\begin{table}[H]
\centering
\small
\setlength{\tabcolsep}{5pt}
\renewcommand{\arraystretch}{1.08}
\begin{tabular}{@{}rr|rr|rr|rr@{}}
\toprule
$n$ & $c_n$ & $n$ & $c_n$ & $n$ & $c_n$ & $n$ & $c_n$ \\
\midrule
1 & 3    & 6 & 160  & 11 & 13440  & 16 & 1143585 \\
2 & 5    & 7 & 390  & 12 & 32735  & 17 & 2781070 \\
3 & 10   & 8 & 940  & 13 & 79610  & 18 & 6762990 \\
4 & 25   & 9 & 2270 & 14 & 193480 & 19 & 16445100 \\
5 & 64   & 10 & 5515   & 15 & 470306 &20&39987325\\
\bottomrule
\end{tabular}
\caption{The remaining coefficients $c_m$ needed in the proof of Proposition \ref{prop:pos}.}
\label{tab:small-c}
\end{table}

This proves $c_m>0$ for every $m\geq 0$, and hence
Proposition \ref{prop:pos}.

%%\clearpage
\section{\label{appxB}Solution to Challenge 3}\label{app:ono-challenge3}

\begin{center}
\textsc%
%\author
{Ken Ono}
\end{center}
%\address{Axiom Math,
%124 University Avenue,
%Palo Alto, CA 94301, USA}
%\email{ken@axiommath.ai}

\noindent
\subsection{Positivity}
This appendix was prepared 
%by Ken Ono 
with assistance from AxiomProver. It gives a combinatorial solution to Challenge~\ref{cha:structural-positivity}. Its essential idea is to interpret an auxiliary sequence as a signed enumeration of colored row-marked rectangles and then to cancel the negative objects by an explicit injection.

\medskip

Let
\[
  S(t):=\sum_{m\ge 0}\sigma(m+1)(-t)^m
       =1-3t+4t^2-7t^3+6t^4-12t^5+\cdots,
\]
and write
\[
  C(t):=
\frac1{S(t)}=\sum_{m\ge0}c_m t^m.
\]
Set
\[
  \auxftn \left( t\right) :=\frac{1+t}{1-t}=1+2t+2t^2+2t^3+\cdots
\]
and
\[
  \auxprodftn (t):=1-\auxftn (t)S(t)=\sum_{n\ge1}\coefauxprod _n t^n.
\]
Then
\begin{equation}\label{eq:ono-C-PQ}
  C(t)=
\frac1{S(t)}=\frac{\auxftn (t)}{1-\auxprodftn(t)}.
\end{equation}
The first values of the auxiliary coefficients are
\[
\begin{array}{c|cccccccccc}
 n   &1&2&3&4&5&6&7&8&9&10\\ \hline
 \coefauxprod _n &1&0&3&4&10&14&21&23&28&34
\end{array}
.
\]
Thus it suffices to prove that all coefficients of $\auxprodftn(t)$ are nonnegative. Indeed, if $\auxprodftn(t)\in t\mathbb Z_{\ge0}[[t]]$, then formally
\[
  \frac1{1-\auxprodftn(t)}=\sum_{j\ge0}\auxprodftn (t)^j\in\mathbb Z_{\ge0}[[t]],
\]
and multiplication by $\auxftn (t)$, whose coefficients are all strictly positive, gives $c_m>0$ for every $m\ge0$.

\begin{theorem}\label{thm:ono-q-nonnegative}
For every $n\ge1$, one has $\coefauxprod _n\ge0$.
\end{theorem}

For $m\ge1$, let $\rectangle _m$ be the set of triples
\[
  (a,b,i) \quad\text{with}\quad a,b\ge1,\quad ab=m,
  \quad 1\le i\le a.
\]
We view $(a,b,i)$ as an $a$-by-$b$ rectangle with its $i$th row marked. Hence
\begin{equation}\label{eq:ono-Rm-sigma}
  |\rectangle _m|=\sum_{a\mid m}a=\sigma(m),
\end{equation}
since choosing the height $a\mid m$ determines the width $b=m/a$, and then there are $a$ possible marked rows.

For $N\ge2$, define the colored finite set
\[
  X_N:=\{(a,b,i,\varepsilon): ab\le N,
      \ 1\le i\le a,
      \ \varepsilon\in \colorset _N(ab)\},
\]
where
\[
  \colorset _N(r):=
  \begin{cases}
    \{0,1\}, & r<N,\\
    \{0\},   & r=N.
  \end{cases}
\]
Thus rectangles of area $<N$ appear in two colors, whereas boundary rectangles of area $N$ appear only in color $0$. Let $X_N^+$ be the subset of objects of even area and $X_N^-$ the subset of objects of odd area.

\begin{lemma}\label{lem:ono-signedcount}
For $n\ge1$, with $N=n+1$,
\[
  \coefauxprod _n=|X_N^+|-|X_N^-|.
\]
\end{lemma}

\begin{proof}
Since $\auxftn (t)=1+2t+2t^2+\cdots$, we have
\[
  [t^n]\auxftn (t)S(t)
  =(-1)^n\sigma(n+1)+2\sum_{m=0}^{n-1}(-1)^m\sigma(m+1).
\]
Because $\auxprodftn (t)=1-\auxftn (t)S(t)$, for $n\ge1$ we obtain, after putting $N=n+1$ and replacing $m+1$ by $r$,
\begin{equation}\label{eq:ono-q-signed}
  \coefauxprod _n= -[t^n]\auxftn (t)S(t)
     =(-1)^N\sigma(N)+2\sum_{r=1}^{N-1}(-1)^r\sigma(r).
\end{equation}
By \eqref{eq:ono-Rm-sigma}, $\sigma(r)$ counts row-marked rectangles of area $r$. In \eqref{eq:ono-q-signed}, each area $r<N$ occurs with two colors, while the boundary area $N$ occurs with one color. The sign is $+1$ for even area and $-1$ for odd area. This is exactly $|X_N^+|-|X_N^-|$.
\end{proof}

\begin{lemma}\label{lem:ono-injection}
For every $N\ge2$, there is an explicit injection
\[
  \Phi_N:X_N^-\hookrightarrow X_N^+.
\]
Consequently, $\coefauxprod _n\ge0$ for all $n\ge1$.
\end{lemma}

\begin{proof}
Let $(a,b,i,\varepsilon)\in X_N^-$. Then $ab$ is odd, so both $a$ and $b$ are odd. We define $\Phi_N$ by the following three rules.

\emph{Case 1: $b>1$.} Define
\[
  \Phi_N(a,b,i,\varepsilon):=(a,b-1,i,\varepsilon).
\]
Since $b-1$ is positive and even, the new area $a(b-1)$ is even. Moreover $a(b-1)<ab\le N$, so both colors are allowed.

\emph{Case 2: $b=1$ and $a<N$.} Here the object is an odd column below the boundary, so both colors are allowed. Define
\[
\begin{aligned}
  \Phi_N(a,1,i,0)&:=(a+1,1,i,0),\\
  \Phi_N(a,1,1,1)&:=(a+1,1,a+1,0),\\
  \Phi_N(a,1,i,1)&:=(a-1,1,i-1,1) \qquad (i>1).
\end{aligned}
\]
Each image is a valid positive column: its height is even, the marked row is valid, and in the last line the area is $<N$, so color $1$ is permitted.

\emph{Case 3: $b=1$ and $a=N$.} This can occur only when $N$ is odd, and then the boundary rule forces $\varepsilon=0$. Define
\[
\begin{aligned}
  \Phi_N(N,1,i,0)&:=(N-1,1,i-1,1) \qquad (i>1),\\
  \Phi_N(N,1,1,0)&:=(1,N-1,1,1).
\end{aligned}
\]
Both images have even area $N-1<N$, so color $1$ is permitted.

It remains to check injectivity. Images from Case 1 are recovered uniquely by increasing the width by $1$. The exceptional image $(1,N-1,1,1)$ from Case 3 cannot be a Case 1 image, since its only possible Case 1 preimage would be $(1,N,1,1)$, which is forbidden by the boundary color rule.

For column images, the inverse is also unique. A color $0$ column of even height $e$ arises either from $(e-1,1,i,0)$ if its marked row is $i<e$, or from $(e-1,1,1,1)$ if the marked row is $e$. A color $1$ column of even height $e<N-1$ arises uniquely from $(e+1,1,i+1,1)$. A color $1$ column of height $N-1$ arises from the boundary rule $(N,1,i+1,0)$, except for the already separated exceptional image $(1,N-1,1,1)$. These alternatives are disjoint. Hence no two negative objects have the same image.

Thus $\Phi_N$ is injective. By Lemma~\ref{lem:ono-signedcount},
\[
  \coefauxprod _n=|X_{n+1}^+|-|X_{n+1}^-|\ge0,
\]
and Theorem~\ref{thm:ono-q-nonnegative} follows.
\end{proof}

\begin{corollary}\label{cor:ono-c-positive}
For every $m\ge0$, the coefficient $c_m$ of $C(t)=
1/S(t)$ is strictly positive.
\end{corollary}

\begin{proof}
By Theorem~\ref{thm:ono-q-nonnegative}, $\auxprodftn (t)\in t\mathbb Z_{\ge0}[[t]]$. Hence
\[
  \frac1{1-\auxprodftn (t)}=1+\auxprodftn (t)+\auxprodftn (t)^2+\cdots
\]
has nonnegative coefficients. Since $\auxftn (t)=1+2t+2t^2+\cdots$ has strictly positive coefficients, \eqref{eq:ono-C-PQ} implies $c_m>0$ for all $m\ge0$.
This gives an independent combinatorial proof of Proposition~\ref{prop:pos}.
\end{proof}

\begin{remark*}
The final step is the standard product--sequence principle for ordinary generating functions: $\auxftn (t)/(1-\auxprodftn (t))$ counts one $\auxftn $-object followed by a finite ordered word of $\auxprodftn $-objects. In the terminology of Flajolet--Sedgewick, this is the symbolic translation of $\mathcal P\times\operatorname{SEQ}(\mathcal Q)$; see \cite[Chapter~I]{FS09}.
\end{remark*}
%%%%%%%%%%%%%%%%%%%%%%%%%%%%%%%%%%%%%%%%%%%%%%%%%%%%%%%%%%%%%%%%%%%
%%%%%%%%%%%%%%%%%%%%%%%%%%%%%%%%%%%%%%%%%%%%%%%%%%%%%%%%%%%%%%%%%%%
\subsection{Lean verification}\label{sec:AxiomProver}

Here we provide the context for the proof of Challenge~3 and record the protocol by which the
structural positivity argument was discovered and checked in Lean by AxiomProver, an AI system that is currently under development.  This Appendix is motivated, in
part, by the broader question of whether an AI system can help find a conceptual explanation for
a positivity phenomenon that is already accessible by analytic or computational means as obtained in this paper.  In the
present setting, the relevant positivity is the assertion that the coefficients $c_m$ defined by
\[
C(t):=  \sum_{m\geq 0} c_m t^m
  =
  \left(\sum_{m\geq 0}\sigma(m+1)(-t)^m\right)^{-1}
\]
are positive for every $m\geq 0$, where
\[
  S(t):=\sum_{m\ge 0}\sigma(m+1)(-t)^m
       =1-3t+4t^2-7t^3+6t^4-12t^5+\cdots .
\]

The proof was developed through human-AI collaboration. The human author recognized that the key is to study the coefficients of
the remainder 
$$Q(t)=1-P(t)S(t)=:\sum_{n\geq 0} q_n t^n.
$$
The role of $P(t)$ is to make the signed
coefficient formula for $Q(t)$ amenable to a finite cancellation model: after the change of
variables $N=n+1$, the coefficient $q_n$ becomes a signed count of row-marked rectangles of
area at most $N$, with a boundary color convention at area $N$. This is Lemma~B.2, which was autonomously proved and verified in Lean by AxiomProver.

The main search problem was then reduced to finding a uniform injection from the negative
objects, corresponding to odd-area rectangles, into the positive objects, corresponding to even-area
rectangles.  The resulting injection is the content of Lemma~B.3 which was autonomously proved and verified in Lean by AxiomProver.

\subsection*{Process}
The formal proofs provided in this work were developed and verified using Lean~4.28.0.
Compatibility with earlier or later versions is not guaranteed due to the
evolving nature of the Lean 4 compiler and its core libraries.
The relevant files are all posted in the following repository:
\begin{center}
  \url{https://github.com/AxiomMath/challenge_3}
\end{center}
The formalization consisted of two separate runs, named
\texttt{lemma-b2} and \texttt{lemma-b3}\footnote{These statements are referred to as Propositions 1 and 2 in GitHub.}.  The run \texttt{lemma-b2} formalizes
Lemma B.2; the run \texttt{lemma-b3} formalizes Lemma B.3.
For each run, the input files, collected under \texttt{input/<run>/}, were the following.
\begin{itemize}
\item \texttt{problem.md}: a natural-language description of the problem to be formalized.
\item \texttt{challenge3\_Part1.tex} or \texttt{challenge3\_Part2.tex}: the \LaTeX{} statement and proof that the formalization was asked to follow.
\end{itemize}
For the run \texttt{lemma-b3}, the \texttt{problem.md} additionally supplied the
formalizations of Lemma B.2 and Lemma B.3 from the \texttt{lemma-b2} run.
Given these files,
AxiomProver produced, for each run, the following output files, collected under \texttt{Challenge\_3/<run>/}.
\begin{itemize}
 \item \texttt{problem.lean}: a translation of the problem statement into Lean.
 \item \texttt{solution.lean}: the formal, machine-checked solution in Lean.
\end{itemize}
For both runs \texttt{lemma-b2} and \texttt{lemma-b3},
both files were generated autonomously by AxiomProver.

After the formal solutions were generated, the human authors wrote this paper
(without the use of AI) for human readers.
At first glance, the proofs found by AxiomProver may not resemble the narrative presented in this paper.
Turning a Lean file into a human-readable proof is difficult
because Lean is written as code for a type-checker.

\section*{Declaration of AI-assistance in the manuscript preparation process}
As described in Appendix B,
AxiomProver (an AI tool under development)
was used to produce formal proofs of Lemma B.2 and Lemma B.3.
The natural language text of the paper, including the appendices, was written without assistance from AI.
\ \\

%%%%%%%%%%%%%%%%%%%%%%%%%%%%%%%%%%%%%%%%%%%%%%%%%%%%%%%%%%%%%%%%%%

\noindent
{\bf Acknowledgments.}
Bernhard Heim and Markus Neuhauser thank Christian Stump for a useful discussion on the zero distribution of polynomials. Ken Ono thanks Simon Mahns and Jujian Zhang for their assistance with Appendix B.

\clearpage
\noindent\textsc{Appendix B author address}

\medskip
\noindent
\textsc{Ken Ono}\\
Axiom Math\\
124 University Avenue\\
Palo Alto, CA 94301, USA\\
\emph{Email address:} \texttt{ken@axiommath.ai}

\end{document}